\documentclass[12pt]{amsart}
\usepackage[utf8]{inputenc}
\usepackage[english]{babel}
\usepackage[T1]{fontenc}
\usepackage[utf8]{inputenc}
\usepackage{titlesec}
\usepackage{tikz}
\usetikzlibrary{arrows.meta}

\usepackage{tikz-cd}
\usepackage{graphicx}
\usepackage{amsmath,amssymb,amsthm}
\usepackage{wrapfig}
\usepackage{subcaption}
\usepackage{enumitem}
\usepackage{hyperref}
\usepackage{ragged2e}
\usepackage{xcolor}
\usepackage{mathrsfs}

\titleformat{\section}
  {\normalsize\bfseries\filcenter}
  {\thesection.}
  {0.5em}
  {}

\newtheorem{topo}{Theorem}[section]
\newtheorem{coro}[topo]{Corollary}
\newtheorem{df}[topo]{Definition}

\newtheorem{fakt}[topo]{Fact}

\newtheorem{lem}[topo]{Lemma}

\newcommand{\per}{\textit{Perf}}
\newcommand{\re}{\!\restriction\!}
\newcommand{\eme}{\mathcal{EM}}
\newcommand{\en}{\mathcal{EN}}

\newcommand{\om}{2^\omega}
\newcommand{\w}{\omega}

\newcommand{\sub}{\subseteq}

\newcommand{\sm}{\!\setminus\!}
\newcommand{\s}{\sigma}

\newcommand{\ww}{[\w]^\w}

\newcommand{\dm}{\text{dom}}
\usepackage{secdot}
\newcommand{\Io}{\mathcal{I}_0}
\newcommand{\ros}{\mathfrak{P}_2}
\newcommand{\omm}{2^{<\omega}}

\newcommand{\res}{\!\restriction\!}
\newcommand{\frr}{{}^\frown}
\newcommand{\fr}{^\frown}
\newcommand{\ps}{\text{p}\_\text{splits}^S(}
\renewcommand{\t}{\tau}

\newcommand{\ee}{\mathcal{EE}}
\newcommand{\bor}{\mathscr{B}}

\newcommand{\ii}{\mathcal{I}}
\newcommand{\sj}{\mathfrak{MK}}

\title{Everywhere $\mathcal{I}$ sets}

\author[J. Cieplechowicz]{Jakub Cieplechowicz}
\email{jakub.cieplechowicz@pwr.edu.pl}

\author[Sz. Żeberski]{Szymon Żeberski}
\email{szymon.zeberski@pwr.edu.pl}
\address[J. Cieplechowicz, Sz. Żeberski]{Faculty of Pure and Applied Mathematics, Wrocław University of Science and Technology, Wybrzeże Stanisława Wyspiańskiego 27, 50-370 Wrocław, Poland}

\subjclass{03E05; 54H05.}
\keywords{Cantor space, ideal, tree, perfect set, Borel set}

\begin{document}

\begin{abstract}
%\noindent\textbf{Abstract.} 
Let $\mathcal{I}$ be a $\s$-ideal on the Cantor space $\om$. A set $X\sub\om$ is called everywhere $\mathcal{I}$ if for any $S\in\ww$ the set $X\re S\in\mathcal{I}(2^S).$ We will discuss relations between such families in the class of perfect and Borel sets.
\end{abstract}

\maketitle

\section{Introduction}

In \cite{ros} Andrzej Rosłanowski introduced the following $\s$-ideal on the Cantor space $\om$:
\[
\ros=\{X\sub\om:(\forall S\in\ww)\ X\res S\neq 2^S\}.
\]
where $X\res S=\{x\res S:x\in X\}.$ This ideal is connected with Mycielski ideals, see \cite{myci}.

As for any $X\in\ros$ its restriction on \textit{every} $S\in\ww$ is not the whole $2^S$, we call sets from $\ros$ \textit{everywhere not everything} sets. Replacing `\(X\res S\neq 2^S\)' in the definition of \(\ros\) by `\(X\res S\) is countable', we obtain the family of \textit{everywhere countable} sets. It can be observed that a set is everywhere countable if and only if it is countable. Thus we obtain nothing new. Miroslav Repick$\acute{\text{y}}$ in \cite{repic} modified the definition of everywhere countable sets defining a $\s$-ideal:
\[\Io = \{X\sub\om:(\forall S\in\ww)(\exists A\in[S]^\w)\ |X\res A|\leq\w\}.
\]

In this paper, we will work in the Cantor space $\om$. The topology of $\om$ is generated by clopen sets of the form $[\s]=\{x\in\om:\s\sub x\}$ for $\s\in\omm=\bigcup_{n<\w}2^n$. By $\mathcal{M}$ we mean the $\s$-ideal of meager sets. Also, we consider the Cantor space $\om$ with the standard product measure, denoted by $\mu$. By $\mathcal{N}$ we mean the $\s$-ideal of null sets.

Fix $S\in\ww$. There is a unique increasing function $s:\w\to\w$ such that $s[\w]=\{s_i:\ i\in\omega\}=S$. Consider $2^{<S}:=\bigcup_{n\in \w\\}2^{\{s_i:\ i<n\}}$. For $\s\in 2^{<S}$ the sets of the form $[\s]_S=\{x\in 2^S:\s\sub x\}$ are clopens in $2^S$. Furthermore, they generate the topology of $2^S.$ By $\mu_S$ we mean a naturally defined measure $\mu$ on $2^S$.\\
Let $\mathcal{I}$ be a $\s$-ideal on $\om$.  Consider a natural bijection $\tilde{s}:\om\to 2^S$, 
\[\tilde{s}(\{\langle i, a_i\rangle : i\in\w\})=\{\langle s_i, a_i\rangle:i\in\w\}.
\]
%For fixed $S\in\ww$, we will consider $2^{<S}:=\bigcup_{n\in \w\\}2^{\{s_i\in S: i<n\}}$. For $\s\in 2^{<S}$ the sets of the form $[\s]_S=\{x\in 2^S:\s\sub x\}$ are clopens in $2^S$; moreover they generate the topology of $2^S.$ By $\mu_S$ we mean a naturally defined measure $\mu$ on $2^S$. 
%Let $\mathcal{I}$ be a $\s$-ideal on $\om$. Fix $S\in\ww$. Hence, there is a unique increasing function $s:\w\to\w$ such that $s[\w]=S$. Consider a natural bijection $\tilde{s}:\om\to 2^S$, $\tilde{s}(\langle s_i:i\in\w\rangle)=\{\langle s(i), s_i\rangle:i\in\w\}$. 
Let 
\[A\in\mathcal{I}(2^S)\iff \tilde{s}[A]\in\mathcal{I}.\]
Notice that $\mathcal{I}(2^S)$ is a $\s$-ideal $\mathcal{I}$ on $2^S$ and therefore we can define a $\s$-ideal of \textit{everywhere} $\mathcal{I}$ sets:
\[\mathcal{EI}=\{X\sub\om:(\forall S\in\ww)\ X\res S\in \mathcal{I}(2^S)\}.\tag{\textit{$\blacktriangle$}} \]
We can formulate a useful fact:
\begin{fakt}\label{eij}
    For $\s$-ideals $\ii,\mathcal{J}\sub\mathcal{P}(\om)$ we get $\mathcal{E}(\ii\cap\mathcal{J})=\mathcal{E}(\ii)\cap\mathcal{E}(\mathcal{J)}$.
\end{fakt}
\begin{proof}
    Notice that $(\mathcal{I}\cap\mathcal{J})(2^S)=\mathcal{I}(2^S)\cap\mathcal{J}(2^S)$, where $S\in\ww.$ This immediately yields $\mathcal{E}(\ii\cap\mathcal{J})=\mathcal{E}(\ii)\cap\mathcal{E}(\mathcal{J)}$.
\end{proof}

In the fashion of $(\blacktriangle)$, another two ideals, besides $\ros$ and $\Io$, that will be our points of interest, were introduced in \cite{krasz} by Jan Kraszewski:
\begin{align*}
&\eme=\{X\sub\om:(\forall S\in\ww)\ X\res S\ \text{is meager in}\ 2^S\},\\
&\en=\{X\sub\om:(\forall S\in\ww)\ X\res S\ \text{is null in}\ 2^S\}.
\end{align*}
Sets from $\eme$ (respectively $\en$) are called \textit{everywhere meager} (\textit{everywhere null)}.

From the definitions we get $\Io\sub\eme\sub\ros\cap\mathcal{M}$ and $\Io\sub\en\sub\ros\cap\mathcal{N}$.

We will also examine a $\s$-ideal $\mathcal{EE}$, where $\mathcal{E}$ stands for $\s$-ideal generated by closed null subsets of $\om$. It can be shown that $\Io\sub\mathcal{EE}\sub\eme\cap\en$. To sum up, we have the diagram: 
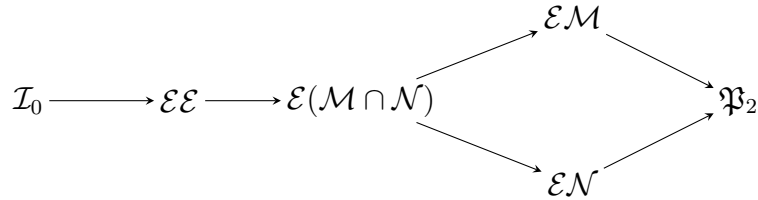
\[
\begin{tikzpicture}[
  node distance=1.8cm and 2.2cm,
  every node/.style={inner sep=2pt},
  arr/.style={->, >=stealth}
]

\node (I0) at (0,0) {$\mathcal I_0$};
\node (UU1) at (2,0) {$\ee$};
\node (UUV) at (4.4,0) {$\mathcal E(\mathcal M\cap\mathcal N)$};

\node (UU2) at (7.2,1.1) {$\eme$};
\node (UV)  at (7.2,-1.1) {$\en$};
\node (B2)  at (9.4,0) {$\ros$};

\draw[arr] (I0) -- (UU1);
\draw[arr] (UU1) -- (UUV);

\draw[arr] (UUV) -- (UU2);
\draw[arr] (UUV) -- (UV);

\draw[arr] (UU2) -- (B2);
\draw[arr] (UV) -- (B2);

\end{tikzpicture}
\]

Where $\rightarrow$ means $\sub$. For the sake of simplicity, define
\[\sj=\{\Io,\ee, \mathcal{E}(\mathcal{M}\cap\mathcal{N}), \eme,\en,\ros\}.\]
An ideal that belongs to $\sj$ is called a $\sj$-ideal.

\vspace{0.5cm}

In \cite{repic} Repick$\acute{\text{y}}$ showed that $\Io$ contains perfect sets, so every $\sj$-ideal also. Our investigations revolve around finding properties that will help us decide whether certain kinds of perfect sets with these properties are in $\s$-ideals of our interest.

To describe perfect sets in $\om$, we will recall some definitions regarding trees. A set $T\sub \omm$ is a \textit{tree} if for each $\tau \in T$ and every $n\in\w$ we have $\tau\res n\in T$. 
\vspace{0.2cm}

\begin{df}
   \normalfont Let $T\sub\omm$ be a tree. Then

\begin{enumerate}
    \item[$\bullet$]$\text{succ}_T(\t)=\{a\in 2:\t\fr a\in T\},$
    
    \item[$\bullet$]$\text{split}(T)=\{\t\in T: |\text{succ}_T(\t)|=2\}, $
 
    \item[$\bullet$] $T$ is {\em perfect} if \[
(\forall\t\in T)(\exists\s\supseteq\t)(\s\in\text{split}(T)),\]
 
 \item[$\bullet$] $T$ is {\em uniformly perfect} if it is perfect and \[(\forall n\in\w)(2^n\cap T\sub \text{split}(T)\ \vee\ 2^n\cap \text{split}(T)=\emptyset),\]

 \item[$\bullet$] $T$ is a {\em splitting tree} if \[(\forall \t \in T)(\exists N \in \omega)(\forall n \ge N)(\forall i \in 2)
(\exists \t' \in T \cap 2^{n+1})(\t \subseteq \t'\ \land\ \t'(n) = i).\]
\end{enumerate}

Let us remark that each splitting tree is perfect. Splitting trees were investigated before, e.g., in \cite{ac}, \cite{hespi}, \cite{lagmil}, \cite{szst1}, \cite{szst2}. 
\end{df}

A body of a tree $T\sub 2^{<\w}$ is the set \[[T]=\{x\in \om :(\forall n\in\w)\, x\re n\in T\}\] of infinite branches of $T$. Bodies of perfect trees are perfect subsets of $\om$.

In the next section, we check what types of perfect sets are in $\sj$-ideals. We also give examples of perfect sets that belong to or omit the ideals from $\sj$.

Furthermore, M. Repick$\acute{\text{y}}$ in \cite{repic} found out that there is no analytic set in $\ros\sm\Io$, i.e., we cannot distinguish the $\sj$-ideals in the class of analytic sets. In the last section, we showed different proofs of the equality of the considered classes of ideals based on our results concerning perfect sets and the properties of particular $\sj$-ideals.

\section{Trees}

Let us denote by $\mathcal{UP}\sub \mathcal{P}(\omm)$ the class of all uniformly perfect trees on $\omm$.

\begin{topo}\label{up}
A body of a uniformly perfect tree is not in $\ros$.
     
\end{topo}
\begin{proof}
     Take an arbitrary $T\in\mathcal{UP}$. Define 
\[ S=\{n\in\w : 2^n\cap T\sub \text{split}(T)\}.\]

Then $|S|=\w$ because $T$ is uniformly perfect. Moreover, $[T]\res S=2^{S}.$ So $[T]\notin \ros$.
\end{proof}

 The `pattern of splits' in a uniformly perfect tree leads us out of $\ros$ and then so out of every $\sj$-ideal.

\begin{df}
    \normalfont A tree $T\sub\omm$ has \textit{unique splits} if \[\forall n\in\w\quad |2^n\cap\text{split}(T)|\leq 1.\]
\end{df}

Denote by $\mathcal{US}\sub\mathcal{P}(\omm)$ the class of all perfect trees with unique splits. We can notice that $\mathcal{UP}\cap\mathcal{US}=\emptyset$.

One remark is that, we will write $\t_{001}\in\omm$ instead of $\t_{\langle 001\rangle}\in\omm$, where $\langle 001\rangle\in\omm$ is just an example of an index.

\begin{topo}
   There exists a perfect tree with unique splits whose body is in $\mathcal{I}_0$. 
\end{topo}
\begin{proof}
   First, for $\s\in\omm$ let $n_\s=(1\fr\s)_2=\sum_{i\in\dm(\s)}2^i\cdot \s(i)+ 2^{|\s|}$ be an interpretation of $1\fr\s$ as a natural number, e.g. $n_{10}=(1\fr10)_2=2^0\cdot0+2^1\cdot1+2^2\cdot1=6$.

   We will construct a sequence $\langle \tau_\s : \s\in \omm\rangle$, where each $\tau_\s\in\omm$, such that:
    \begin{enumerate}
        \item[(1)] $\tau_\emptyset=\emptyset$,
        \item[(2)] $(\forall\s)$\quad  $\tau_{\s^\frown i}=\tau_\s{}^\frown\underbrace{0\frr\dots\frr0}_{n_\s}\frr i$.
        %\item[(3)] $|t_\s|=|t_\eta|\iff (\exists\tau) (\tau\fr1=\s\ \wedge\ \tau\frr0=\eta)$.
    \end{enumerate}

At first, put $\tau_\emptyset=\emptyset$. Next, following $n_\emptyset=(1\fr\emptyset)_2=1$ let 
\begin{align*}
&\tau_0=\tau_{\emptyset\fr 0}=\tau_\emptyset\frr0\fr\dots\fr 0\fr 0=00,\\ &\tau_1=\tau_{\emptyset\fr1}=\tau_\emptyset\frr0\fr\dots\fr0\fr1=01. 
\end{align*}

We can notice that $\tau_\emptyset\subseteq \tau_0, \tau_1$. Now, as a reason of $n_0=(1\fr0)_2=2$, we take $\tau_{00}=\tau_{0\fr0}=\tau_0\frr00\fr0=00 000$. Similarly define $\tau_{11},\tau_{10}$ and $\tau_{01}$. We get that $\tau_0\sub \tau_{00},\tau_{01}$ and $\tau_1\sub \tau_{11}, \tau_{10}$. Suppose that we already have $\tau_\eta$, where $\eta\in 2^n$, with the above properties. So, there is $\s\in2^{n+1}$ such that $\s=\eta\fr i$, $i\in 2$. In the same fashion as before, put

\[
\tau_\s=\tau_{\eta\fr i}=\tau_\eta{}^\frown\underbrace{0\frr\dots\frr0}_{n_\eta}\frr i.
\]

By the recursion, we obtain $\langle \tau_\s : \s\in \omm\rangle$. Let $T=\{\tau_\s\in\omm :\s\in\omm\}$ be our tree. We also get that $\t_\s\sub\t_\xi$ if $\s\sub\xi$. Thus, we can observe that $T$ is indeed perfect. 
\vspace{0.3cm}

The following property will be crucial in showing that $[T]\in\Io$:

\[\mathbf{(\ast)} \quad |\tau_\s|=|\tau_\eta|\iff (\exists\mu) (\mu\fr1=\s\ \wedge\ \mu\frr0=\eta).\]

To verify that $T$ satisfies it, first observe that

 \[
 |\tau_\s|=|\tau_{\eta\fr i}|=|\tau                                                                                   _\eta|+n_\eta +1=\sum_{\nu\subsetneq\s }n_\nu +|\s|.
 \]

For any two $|\xi|<|\xi'|$ we want $\sum_{\nu\subsetneq\xi}n_\nu<\sum_{\nu'\subsetneq\xi'}n_{\nu'}$. Consider the worst case, namely $\xi\equiv1$ and $\xi'\equiv0$. Now see 
\begin{align*}
    |\tau_\xi|&=\sum_{\nu\subsetneq\xi}n_\nu + |\xi|=2^{|\xi|+1}-2-|\xi|+|\xi| \\&=2^{|\xi|+1}-2< 2^{|\xi'|}-1+|\xi'|=\sum_{\nu'\subsetneq\xi'}n_{\nu'}+|\xi'|=|\tau_{\xi'}|.
\end{align*}

Even in that case we have $|\tau_\xi|<|\tau_{\xi'}|$. Next let $|\tau_\s|=|\tau_\eta|$ for $\s\neq\eta$. If $|\s|\neq|\eta|$ then we have $|\tau_\s|\neq|\tau_\eta|$, so $|\s|=|\eta|$. Therefore, $\sum_{\nu\subsetneq\s}n_\nu=\sum_{\nu'\subsetneq\eta}n_{\nu'}$. That means 
\[(\forall\nu)(\nu\subsetneq\s\iff\nu\subsetneq\eta),\]
but $\s\neq\eta$, so they differ on the last place, then so the property $(\ast)$ is satisfied. Therefore, $T\in\mathcal{US}$. We can also make an observation that `split means having 1 as a value' in $T$.

   To show that $[T]\in\mathcal{I}_0$, pick an arbitrary $S\in[\w]^\w$. Let $S=\{s_i : i\in\w\}$, $i<j\iff s_i<s_j$. We want to find $\{a_i: i\in\w\}\in [S]^\w$ such that $[T]\res \{a_i: i\in\w\}$ is countable. We will focus on splitting places.

   To start off, define the \textit{possible splits} of $\s\in\omm$ in $S$:
\[\text{p}\_\text{splits}^S(\s)=\{n\in S: (\exists \tau\in 2^n)( \tau\supseteq\s\ \wedge\ \tau\fr1,\tau\fr0\in  T)\}.\]
    
To start off, let $a_0=\text{min}S$ and define $V_0=T\cap 2^{a_0+1}$. Thanks to the property $(\ast)$, for $\s,\s'\in V_0$,  we have 
\[
 \s\neq\s'\implies\ps\s)\cap\ps\s')=\emptyset
\] 
We have two cases to consider, namely:
\vspace{0.5cm}
    
    \textbf{Case 0.} $\ps\emptyset)$ is finite. 
\vspace{0.5cm}

    Then $[T]\res S\sub\{x\in 2^S: (\forall n\in S) (x(n)=1\to n\in\ps\emptyset))\}$, and see that
    \[|\{x\in 2^S: (\forall n\in S) (x(n)=1\to n\in\ps\emptyset))\}|\leq 2^{|\ps\emptyset)|}.\]
Thus, $[T]\res S$ is finite.

\vspace{1cm}

\textbf{Case 1.} $\exists\s_0\in V_0\quad |\ps\s_0)|=\w$.
\vspace{0.5cm}

Put $S_0=\ps\s_0)$ and then  $a_1=\text{min}S_0$. Next, define \[V_1=\{\tau\in T:\tau\supseteq\s_0\ \wedge\ \tau\in 2^{a_1+1}\}.\]

Because we are in Case 1., we know there is $\s_1\in V_1$ such that $|\ps\s_1)|=\w$. Now, in the same fashion as before, put $S_1=\ps\s_1)$ and then let $a_2=\text{min}S_1$. Define \[V_2=\{\tau\in T:\tau\supseteq\s_1\ \wedge\ \tau\in 2^{a_2+1}\}.\]
Suppose that we have $S_n$. Put $a_{n+1}=\text{min}S_n$ and then \[V_{n+1}=\{\tau\in T:\tau\supseteq\s_n\ \wedge\ \tau\in 2^{a_{n+1}+1}\}.\]

 Thanks to the recursion, we get the set $A=\{a_i:i<\w\}\sub S$. Also, the construction provides that $|A|=\w$.

We can make an observation that 
\[[T]\res A=\bigcup_{n\in\w}\bigcup_{\s\in V_n\setminus\{\s_{n}\}}\big([T]\cap[\s]\big)\res A\cup\{\bigcup_n\s_n\res A\}.\]

Take any $\s\in V_n\sm\{\s_n\}$. Next, notice 
\[\tau\in([T]\cap[\s])\res A\implies (\forall k\in A)(k>a_n\to \tau(k)=0), \]

so the set $([T]\cap[\s])\res A$ is finite. Therefore, $|[T]\res A|\leq \w$, and finally $[T]\in\Io$.

\end{proof}

 For $\s\in T\sub\omm$ let $N_\s$ be  a minimal number such that
    \[(\forall n\geq N_\s)(\forall i\in2)(\exists \s'\in T\cap2^{n+1})(\s\sub\s' \wedge\s'(n)=i).\]
    We call that $N_\s$ a \textit{splitting point} for $\s$.

\begin{topo}
    There exists a perfect tree with unique splits whose body is not in $\ros$.
\end{topo}
\begin{proof}
    We will construct a sequence $\langle\tau_\s : \s\in\omm\rangle$, where each $\tau_\s\in\omm$, in such a way that
\begin{enumerate}
\item[(1)] $\tau_\emptyset=\emptyset$,
        \item[(2)] For each $\s$ we have $\tau_{\s^\frown i}=\tau_\s{}^\frown\underbrace{i\frr\dots\frr i}_{n_\s+i}$, where $n_\s=(1\fr\s)_2$. 
        %\item[(3)] $|t_\s|=|t_\eta|\iff \s=\eta$.
\end{enumerate}
\vspace{0.5cm}

First, put $\tau_\emptyset=\emptyset$. Now, because $n_\emptyset+1=2$ and $n_\emptyset+0=1$ we set respectively
\begin{align*}
&\tau_1=\tau_{\emptyset\fr 1}=\tau_\emptyset\frr1\fr\dots\fr1=11,\\ 
&\t_0=\t_{\emptyset\fr0}=\t_\emptyset\frr0\fr\dots\fr0=0.
\end{align*}
We can notice that $\t_\emptyset\sub\t_1, \t_0$. In the same fashion, define $\t_{00},\t_{01}, \t_{10}$ and $\t_{11}$. We also get that $\t_0\sub \t_{00}, \t_{01}$ and $\t_1\sub \t_{10}, \t_{11}$.  Suppose that we have already set $\t_\eta$ for $\eta\in 2^n$. So, there is $\s\in2^{n+1}$ such that $\s=\eta\fr i$, $i\in 2$. As before, put
\[
\t_\s=\t_{\eta\fr i}=\t_\eta{}^\frown\underbrace{i\frr\dots\frr i}_{n_\s+i}.
\]

By the recursion, we obtain $\langle \tau_\s : \s\in \omm\rangle$. Let $H=\{\tau_\s\in\omm :\s\in\omm\}$ be our perfect tree. We also get that $\t_\s\sub\t_\xi$ if $\s\sub\xi$. So, $H$ is a perfect tree.
\vspace{0.3cm}

To show that this construction works, we will prove that $H$ satisfies

\[\mathbf{(\star)}\quad |\t_\s|=|\t_\eta|\iff \s=\eta.\]

First, see that
\[|\t_\s|=|\t_{\eta\fr i}|=|\t_\eta|+n_\eta+i=\sum_{\nu\subsetneq\s}n_\nu\ +\sum_{k\in \text{dom}(\s)}\hspace{-0.3cm}\s(k).\]

Let us pick $|\s|>|\eta|$ such that $|\t_\s|=|\t_\eta|$. Clearly, $\s\neq\eta$. Let us consider `the worst scenario', namely $\s\equiv0$ and $\eta\equiv1$. Thus \[|\t_\s|=\sum_{\delta\subsetneq\s}n_{\delta}=2^{|\s|}-1.\] But for $\eta$ we get 
\[|\t_\eta|=\sum_{\delta'\subsetneq\eta}n_{\delta'}+\sum_{k\in \text{dom}(\eta)}\eta(k)=2^{|\eta|+1}-2-|\eta|+|\eta|=2^{|\eta|+1}-2.\] 

So, we obtain $|\t_\eta|<|\t_\s|$ if $|\eta|<|\s|$. Next, take $\s\neq\eta$ such that $|\s|=|\eta|$ and $|\t_\s|=|\t_\eta|$. So there is a smallest $N$ for which $\s(N)\neq\eta(N)$. Assume, without losing the generality, that $\s(N)=1$. We aim to show $|\t_\s|-|\t_\eta|>0$. The worst that can happen is 
\begin{align*}&\s(n)=0\ \text{for all}\ n>N\\
&\eta(n)=1\ \text{for all}\ n>N.
\end{align*}
Now, we can only consider shorter lengths:
\[|\t_\s|_N:=\sum_{\substack{\delta\subsetneq\s \\ |\delta|\geq N+1}}\hspace{-0.2cm}n_\delta +\sum_{n\geq N }\s(n).\]
Make a subtraction
\begin{align*}
    |\t_\s|_N-&|\t_\eta|_N=\bigg(\sum_{\substack{\delta\subsetneq\s \\ |\delta|\geq N+1}}\hspace{-0.2cm}n_\delta -\sum_{\substack{\delta'\subsetneq\eta \\ |\delta'|\geq N+1}}\hspace{-0.2cm}n_{\delta'}\bigg)+\bigg(\sum_{n\geq N }\s(n)-\sum_{n\geq N }\eta(n)\bigg)\\=&2^{|\s|-N}
-1+0-2^{|\eta|-N}+1+|\eta|-N+1-|\eta|+N=1>0. \end{align*}
So, $|\t_\s|-|\t_\eta|>0$, a contradiction. Therefore, $\s=\eta$ and the property $(\star)$ is satisfied. Hence, $H\in\mathcal{US}$.

 We will construct a set $\{a_i:i\in\w\}\in\ww$ in such a way that $[H]\res\{a_i:i\in\w\}=2^{\{a_i : i\in\w\}}.$

Take $a_0=0$. Now, take any function $\s_0\in\omm$ such that $\s_0(a_0)=1$. Let $a_1=N_{\s_0}>0$. Next, take $\s_1\supseteq\s_0$ such that $\s_1(a_1)=1$ and $N_{\s_1}>N_{\s_0}$. Put $a_2=N_{\s_1}$. Suppose that we have constructed  $\langle a_i:i\leq n\rangle$ for some $n\in\w$. Pick such $a_{n+1}=N_{\s_n}$ that $N_{\s_n}>N_{\s_m}$ for all $m<n$ and $\s_{n}\sub\s_{n+1}$, where $\s_{n+1}(a_{n+1})=1$. Let $A=\{a_i:i\in\w\}.$  Certainly, $|A|=\w$.

To see that set works, take an arbitrary $a\in A$. Then $a=N_{\s_x}$ for some $\s_x\in\omm$ and $x\in\w$. Then $N_{\s_x}$ is the maximal splitting point at level $2^{N_{\s_x}+1}$. Thus, for all $f\in2^{\leq N_{\s_x}+1}$ and every $j\in2$ we find $f'\supseteq f$ such that $f'(N_{\s_x})=j$. Hence,  $[H]\re A=2^{A}$, so $[H]\notin\ros$.
\end{proof}

Thus, a (body of a) tree from $\mathcal{US}$ can be in every $\sj$-ideal or can omit all of them.

Furthermore, we notice that $H$ is a splitting tree, and it turns out we can generalize the previous theorem. Let $\mathcal{ST}\sub\mathcal{P}(\omm)$ be a class of all splitting trees on $\omm$.

\begin{topo}
    For any $T\in\mathcal{ST}$ we have $[T]\notin\ros$.
\end{topo}
\begin{proof}
Fix $T\in\mathcal{ST}$. We are going to construct a set  $\{ s_i:i\in\w\}\in\ww$ such that $[T]\res \{s_i:i\in\w\}=2^{\{s_i:i\in\w\}}$.

At the very first step put $s_0=0$. Then for $\s_\emptyset=\emptyset$ put $s_1\geq N_{\emptyset}$, $s_1>s_0$. So, there are $\s_1,\s_0\in T$, $\s_1, \s_0\supseteq \s_\emptyset$ such that $\s_1(s_1)=1$, $\s_0(s_1)=0$. Then, take $s_2\geq\text{max}\{N_{\s_1}, N_{\s_0}\}$, $s_2>s_1.$  Suppose that we have taken $s_n$. In the same way put $s_{n+1}\geq \text{max}\{N_{\s_j}: |j|=n\}$ and $s_{n+1}>s_n.$
Let $S=\{s_i :i\in\w\}$. Notice that $|S|=\w$ and $[T]\res S=2^S$, so $[T]\notin\ros$.

\end{proof}
So, every splitting tree omits $\sj$-ideals; the same we can say about uniformly perfect ones. In the case of trees with unique splits, the outcomes may be different.

\section{Perfect and Borel sets}

Let $\per$ be a class of all perfect sets on $\om$.
\vspace{0.3cm}

Make an observation:
\begin{lem}\label{closed}
    For any $P\in \per$ and any $S\in\ww$ the set $P\res S$ is closed.
\end{lem}
\begin{proof}
    Fix $S\in\ww$ and $P\in \per$. Because $\om$ and $2^S\times2^{\w\sm S}$ are homeomorphic, we have $P\res S=\pi_S[P]$. And that $\pi_S[P]$ is a projection of a compact set, so it is closed.
\end{proof}

A straightforward result:
\begin{topo}
    $\eme\cap\per=\ros\cap\per$.
\end{topo}
\begin{proof}
    Pick any $P\in\per\,\sm\eme$. So, there is $S\in\ww$ such that $P\res S \notin\mathcal{M}(2^S)$. Then, because $P\res S$ is closed, we get that $\text{int}_{2^S}(P\res S)\neq \emptyset$. That means, there is $\s\in2^{<S}$ such that $[\s]\sub P\res S$. Therefore, $P\res(S\sm\text{dom}(\s))=2^{S\,\sm\,\text{dom}(\s)}$, hence $P\notin \ros$.

\end{proof}

The following definition on the real line can be found in \cite{oxto}, before Theorem 3.20. 

\begin{df}
   \normalfont A measurable set $E\sub\om$ is said to have \textit{density 1 at} $e\in\om$, if
    \[\frac{\mu(E\cap[e\res n])}{2^{-n}}\to 1,\ \text{if }\ n\to 
    \infty. \]
\end{df}

\vspace{0.5cm}

Recall that the \textit{symmetric difference} of two sets $A$ and $B$ is the set of points that belong to one but not to both of the sets. We denote it by $A\ \Delta\ B$. Thus, 
    \[A\ \Delta\ B=(A\sm B)\cup(B\sm A).\]
   Define $\Phi(E)$ as a set of all density $1$ points of $E\sub\om$.  Theorem 3.20 in \cite{oxto}, formulated in the Cantor space $\om$, says 
   \begin{align}
     \textit{for any measurable set } E\sub\om,\; \mu(E\ \Delta\ \Phi(E))=0. \tag{$\mathfrak{o}$}
   \end{align}
As a matter of fact, $(\mathfrak{o})$ will be useful to obtain the next result.

% \vspace{0.2cm}

\begin{topo}\label{perfen}
    Every perfect set which is not in \(\en\) contains a body of a uniformly perfect tree. Hence, $\en\cap\per=\ros\cap\per.$
\end{topo}
\begin{proof}
    Take any $P\in\per\,\, \sm\en$. So, there is $S\in\ww$ such that $P\res S\notin\mathcal{N}(2^S)$. By \hyperref[closed]{\textbf{Lemma \ref*{closed}}}, $P\res S$ is measurable, hence of positive measure in $2^S$. Then, $\mu_S(P\res S\ \Delta\ \Phi(P\res S))=0$ and so $\Phi(P\res S)\neq \emptyset$. Thus we can take $x_\emptyset\in\Phi(P\res S)$, and there is $\s_\emptyset\in 2^{<S}$, $\s_\emptyset\sub x_\emptyset$ such that \[\mu_S([\s_\emptyset]_S\cap P\res S)>\frac{1}{2}\mu_S([\s_\emptyset]).\]
Thus,  we can observe that
\begin{align*}
    &\mu_S([\s_\emptyset\frr0]_S\cap P\res S)> 0,\\
    &\mu_S([\s_\emptyset\frr1]_S\cap P\res S)> 0.
\end{align*}
   Hence, we have density $1$ points of these two sets, namely pick
   \begin{align*}
       &x_0\in\Phi([\s_\emptyset\frr0]_S\cap P\res S),\\
   &x_1\in\Phi([\s_\emptyset\frr1]_S\cap P\res S).
   \end{align*}
   Now, we are able to find $\s_0,\s_1\in 2^{<S}$,  $|\s_1|=|\s_0|$ such that $\s_0\supseteq\s_\emptyset\frr 0$, $\s_1\supseteq\s_\emptyset\frr 1$, and $\s_0\sub x_0$, $\s_1\sub x_1$. Moreover, we see that
\begin{align*}
    &\mu_S([\s_0]_S\cap P\res S)>\frac{1}{2}\mu_S([\s_0]_S),\\
    &\mu_S([\s_1]_S\cap P\res S)>\frac{1}{2}\mu_S([\s_1]_S).
\end{align*}
Now, one can have
\begin{align*}
    &\mu_S([\s_0\frr0]_S\cap P\res S)> 0,\\
    &\mu_S([\s_0\frr1]_S\cap P\res S)> 0,\\
    &\mu_S([\s_1\frr0]_S\cap P\res S)> 0,\\
    &\mu_S([\s_1\frr1]_S\cap P\res S)> 0.
\end{align*}
   In the same fashion, we can find
   \begin{align*}
       &x_{00}\in\Phi([\s_0\frr 0 ]_S\cap P\res S),\\
       &x_{01}\in\Phi([\s_0\frr 1 ]_S\cap P\res S),\\
       &x_{10}\in\Phi([\s_1\frr 0 ]_S\cap P\res S),\\
       &x_{11}\in\Phi([\s_1\frr 1 ]_S\cap P\res S).
   \end{align*}
Like before, for all $i_1,i_2\in2$ there are $\s_{i_1i_2}\in 2^{<S}$ of the same lengths such that $\s_{i_1i_2}\supseteq \s_{i_1}\frr i_2$ and $\s_{i_1i_2}\sub x_{i_1i_2}$.
Suppose that we have taken $x_\eta\in\Phi(P\res S)$, $\s_\eta\in2^{<S}$, $\s_\eta\sub x_\eta$ for $\eta \in 2^n$.  So, we can find $x_{\eta\fr i}\in\Phi([\s_\eta\frr i]_S\cap P\res S)$ for all $i\in 2$. Then so for all $i\in 2$ pick $\s_{\eta\fr i}\supseteq\s_\eta\frr i$  of the same lengths such that $\s_{\eta\fr i}\sub x_{\eta\fr i}$ and 
\[\mu_S([\s_{\eta\fr i}]_S\cap P\res S)>\frac{1}{2}\mu_S([\s_{\eta\fr i}]_S).\]
That provides us
\begin{align*}
    &\mu_S([\s_{\eta\fr i}\frr0]_S\cap P\res S)> 0,\\
  &\mu_S([\s_{\eta\fr i}\frr1]_S\cap P\res S)> 0
\end{align*}
for all $i\in 2.$ By Recursion we get  $\langle\s_\eta :\eta\in 2^{<\w}\rangle$, where each $\s_\eta\in 2^{<S}$. Let $\Sigma=\{\t\in 2^{<S}:(\exists \eta\in\omm)\ \t\sub\s_\eta\}$. One can observe that $\Sigma\in\mathcal{UP}(2^S)$;  moreover, the set $[\Sigma]$ is closed, so $ [\Sigma]\sub P\res S$. Thus, by \hyperref[up]{\textbf{Theorem \ref*{up}}}, $P\notin \ros.$ 
\end{proof}

\vspace{0.2cm}

\begin{coro}
    $\ee\cap\per=\ros\cap\per$.
\end{coro}
\begin{proof}
  Take any $P\in\per\,\sm\ee$. So, for some $S\in\ww$, $P\res S\notin\mathcal{E}(2^S)$. By \hyperref[closed]{\textbf{Lemma \ref*{closed}}} the set $P\res S$ is closed; hence,  $P\res S$ is of positive measure. Therefore, by repeating the reasoning from \hyperref[perfen]{\textbf{Theorem \ref*{perfen}}}, $P\notin\ros.$
\end{proof}
\vspace{0.5cm}

In addition to the above theorems, consider $\mathcal{E}(\mathcal{M}\cap\mathcal{N})$. By \hyperref[eij]{\textbf{Fact \ref*{eij}}}, $ \mathcal{E}(\mathcal{M}\cap\mathcal{N})=\eme\cap\en$. So:

\begin{coro}
    $\mathcal{E}(\mathcal{M}\cap\mathcal{N})\cap\per=\ros\cap\per$.
\end{coro}

\vspace{0.5cm}

Recall that for every meager set $M\in\mathcal{M}$ there is $x_M\in\om$ and a partition $(I_n)_{n\in\w}$ of $\w$ into finite intervals such that:
\[M\sub\{x\in\om:(\forall^\infty n)\ x\res I_n\neq x_M\res I_n\}.   \]
Moreover, the latter set is meager. The proof of it can be found in \cite[Theorem 2.2.4]{bart}. 

\vspace{0.5cm}

Let us now consider a class $\mathscr{B}$ of all Borel sets.

\begin{topo}\label{embor}
    $\eme\cap\bor=\ros\cap\bor.$
\end{topo}
\begin{proof}
    Let $B\in\bor\sm\eme$. There is $S\in\ww$, $B\res S\notin\mathcal{M}(2^S).$ Notice that $B\res S$ is a projection of a Borel set. So $B\res S$ is analytic and therefore  has the Baire property. Hence, there is an open set $U\sub 2^S$ and $M\in\mathcal{M}(2^S)$ such that $U\sm M\sub B\res S$. By that, there is $\s\in 2^{<S}$, $[\s]_S\sm M\sub B\res S$. Consider $S'=S\sm\dm(\s)$. We can find $(I_n)_{n\in\w}$ a partition of $S'$ into finite intervals and $x_M\in 2^{S'}$ such that
    \[M\sub\{x\in 2^{S'}:(\forall^\infty n)\ x\res I_n\neq x_M\res I_n\}.\]
    One can observe that $M^c\sub B\res S'$. Define $S''=\bigcup_{n\in\w}I_{2n}$. We obtain that $B\res S''=2^{S''}$, so $B\notin\ros$.
\end{proof}
\vspace{0.2cm}

\begin{topo}
    $\en\cap\bor=\ros\cap\bor$.
\end{topo}
\begin{proof}
    Let $B\in\bor\sm\en$. So, there exists $S\in\ww$, $B\res S\notin \mathcal{N}(2^S)$. Once again, $B\res S$ is analytic, thus measurable. By all of that, it contains a closed set of positive measure. So, by Cantor-Bendixson Theorem (\cite[Theorem 4.6]{jech}), $B\re S$ contains a perfect set of positive measure $P\sub 2^S.$ But, by \hyperref[perfen]{\textbf{Theorem \ref*{perfen}}}, there is $S_\heartsuit\in[S]^\w\sub\ww$ such that $P\res S_\heartsuit=2^{S_\heartsuit}$ . Therefore,  $P\notin\ros$, so $B\notin\ros$. 

\end{proof}
\vspace{0.2cm}

\begin{topo}
    $\ee\cap\bor=\ros\cap\bor$.
\end{topo}
\begin{proof}
    Consider any $B\in\bor\sm\ee$. So, there is $S\in\ww$, $B\res S\notin\mathcal{E}(2^S)$. The set $B\res S$ is analytic. Then, by  \cite[Lemma 2.11]{szym}  there is $P\in \per\, (2^S)$, $\mu_S(P)>0$ and a $G_\delta$ set $G\sub P$ which is dense in $P$; moreover, $G\sub B\res S$.  Thus, using \hyperref[perfen]{\textbf{Theorem \ref*{perfen}}}, there exists $S'\in[S]^\w$ such that $P\res S'=2^{S'}.$ Note that $G\res S'\sub B\res S'\sub P\res S'$. We will  show that $G\res  S'$ is a dense $G_\delta$ set in $P\res S'.$  So:
    \begin{itemize}
        \item[$(\bullet)$] $G\res  S'$ is dense in $P\res S'.$

        Take any $\s\in 2^{<S'}$, $[\s]_{S'}\sub P\res S'$. One has $\eta\in 2^{<S}$, $\eta\res S'=\s$ such that $[\eta]_S\cap P\neq \emptyset$. The set $G\sub 2^S$ is dense in $P\sub 2^S$, so we have $x\in[\eta]_S\cap G $. But then $x\res S'\in[\s]_{S'}\cap G\res S'$, and we are done.

         \item[$(\bullet\bullet)$] $G\res S'$ is a $G_\delta$ set in $P\res S'$.

To show the above, divide the argument into the following steps denoted by `$\rightsquigarrow$':

         \begin{itemize}
             \item[$\rightsquigarrow$] For $\s\in 2^{<S}$ we have $[\s]_S\res S'=[\s\res S']_{S'}$.
            
\vspace{0.1cm}
             
             Take any $f\in[\s]_{S}\res S'$. There is $x\in 2^S, x\supseteq\s$ such that $x\res S'=f.$ 
             Notice that $\s\res S'\sub x\res S'.$ Hence, $f\in [\s\res S']_{S'}.$

             Now, take $g\in [\s\res S']_{S'}.$ Define \[\tilde{g}=g\cup \s \cup (S\sm (S'\cup {\rm dom}(\s ))\times\{0\}). \]
             Notice that \(\tilde{g}\in [\s]_S\) and \(\tilde{g}\res S'=g\), hence \(g\in [\s]_S\res S'\).

             \item[$\rightsquigarrow$] $\big(\bigcup_{n<\w}X_n\big)\res S'=\bigcup_{n<\w}X_n\res S'$.
             \vspace{0.3cm}

             Indeed,
             \begin{align*}
            \big(\bigcup_{n<\w}X_n\big)\res S'&=\{x\res S':x\in\bigcup_{n<\w}X_n\big\} =\{x\res S':(\exists n<\w)\ x\in X_n\}\\&=\bigcup_{n<\w}\{x\res S':x\in X_n\}=\bigcup_{n<\w}X_n\res S'.
             \end{align*}

\item[$\rightsquigarrow$]  $\big(\bigcap_{n<\w}X_n\big)\res S'=\bigcap_{n<\w}X_n\res S'$.
             \vspace{0.3cm}

             The argument is similar to that in the case of the generalized sum.

         \end{itemize}

     Go back to our $G\res S'$. First, there are $\s_{n,m}\in 2^{<S}$ such that $G=\bigcap_{n<\w}\bigcup_{m<\w}[\s_{n,m}]_{S}$. Now we can write the following equalities:
     \begin{align*}
         G\res S'=\left(\bigcap_{n<\w}\bigcup_{m<\w}[\s_{n,m}]_{S}\right)\res S'=\bigcap_{n<\w}\bigcup_{m<\w}[\s_{n,m}\res S']_{S'}.
     \end{align*}
    Thus, $G\res  S'$ is also a $G_\delta$ set in $P\res S'.$
    
    \end{itemize}

Recall that $G\res S'\sub B\res S'$. We have shown that $B\res S'$ contains a comeager subset of $2^{S'}$. By \hyperref[embor]{\textbf{Theorem \ref*{embor}}}, $B\notin \ros.$ 
\end{proof}

\end{document}